\documentclass[reqno]{amsart}
\usepackage{amssymb, amsthm, amsmath, amsfonts}
\usepackage{graphics}
\usepackage{hyperref}
\usepackage[all]{xy}
\usepackage{enumerate}
\usepackage[mathscr]{eucal}
\usepackage{thmtools}

\declaretheorem[numberwithin=section]{theorem}
\declaretheorem[name=Lemma, sibling=theorem]{lemma}

\declaretheorem[name=Corollary, sibling=theorem]{corollary}
\declaretheorem[name=Definition, style=definition, sibling=theorem]{definition}

\declaretheorem[name=Remark, style=definition, sibling=theorem]{remark}

\DeclareMathOperator{\Ob}{Ob}
\DeclareMathOperator{\HH}{H}
\DeclareMathOperator{\RR}{R}
\DeclareMathOperator{\reg}{reg}

\newcommand{\kk}{{\mathbf{k}}}
\newcommand{\II}{{\mathscr{I}}}
\newcommand{\FI}{{\mathrm{FI}}}
\newcommand{\VI}{{\mathrm{VI}}}
\newcommand{\VB}{{\mathrm{VB}}}
\newcommand{\UU}{{\mathbf{U}}}
\newcommand{\injdim}{{\mathrm{inj\,dim}}}

\title{Bounds on homological invariants of VI-modules}

\author{Wee Liang Gan}
\address{Department of Mathematics, University of California, Riverside, CA 92521, USA}
\email{wlgan@ucr.edu}

\author{Liping Li}
\address{LCSM (Ministry of Education), School of Mathematics and Statistics, Hunan Normal University, Changsha, Hunan 410081, China}
\email{lipingli@hunnu.edu.cn}

\thanks{L. Li is supported by the National Natural Science Foundation of China 11771135 and the Start-Up Funds of Hunan Normal University 830122-0037. The authors thank Rohit Nagpal for informing us that the approach in \cite{church} cannot directly apply to VI-modules. We also thank the referee for many helpful suggestions to improve the paper.}

\begin{document}

\begin{abstract}
We give bounds for various homological invariants (including Castelnuovo-Mumford regularity, degrees of local cohomology, and injective dimension) of finitely generated VI-modules in the non-describing characteristic case. It turns out that the formulas of these bounds for VI-modules are the same as the formulas of corresponding bounds for FI-modules.
\end{abstract}

\maketitle

\section{Introduction}

Let $\kk$ be a commutative Noetherian ring, and let $F$ be a finite field whose order $q$ is invertible in $\kk$ (called the non-describing characteristic case). The category $\VI$ has the finite dimensional $F$-vector spaces as its objects and the injective $F$-linear maps as its morphisms. By definition, a $\VI$-module is a covariant functor from $\VI$ to the category of $\kk$-modules.

The purpose of this paper is to prove bounds for various homological invariants of finitely generated VI-modules; these homological invariants have been shown to be finite in a recent paper \cite{nagpal} of Nagpal. Surprisingly, though the combinatorial structure and representation theory of VI seem to be more complicated than that of the category FI, whose objects are finite sets and morphisms are injections, homological invariants of $\VI$-modules and $\FI$-modules are bounded by the same formulas. That is, many results on FI-modules proved in \cite{church, ce, gan, li2} can be extended to VI-modules. In particular, we obtain upper bounds of Castelnuovo-Mumford regularity, degrees of local cohomology, and injective dimension of  finitely generated $\VI$-modules $M$, and these upper bounds are in terms of the first two homological degrees of $M$, which measure the degrees of the generators and relations of $M$.

The key ingredients in our arguments are the shift theorem of Nagpal \cite[Theorem 4.34(a)]{nagpal} and (a modification of) the approach for $\FI$-modules described in \cite{li2, ly} (in particular, we avoid an argument in the proof of \cite[Theorem 2.4]{li2} which used a result \cite[Corollary 2.12]{ly} that may not hold for $\VI$-modules). For $\FI$-modules, other approaches have been found by various authors \cite{church, ce,gan}, but at this moment it is not obvious to us if their arguments can also be adapted to $\VI$-modules.

We collect the main results of this paper in the following theorem. For definitions and notations, see the next section.

\begin{theorem}
Let $\kk$ be a commutative Noetherian ring such that $q$ is invertible in $\kk$. Let $M$ be a finitely generated $\VI$-module over $\kk$. Let $t_i(M)$ be the degree of the $i$-th $\VI$-homology $\HH_i^\VI(M)$ of $M$.
Let $h_i(M)$ be the degree of the $i$-th local cohomology $\RR^i \Gamma (M)$ of $M$. Then one has:
\begin{enumerate}[(a)]
\item
 \[ h_i(M) \leqslant
\left\{ \begin{array}{ll}
t_0(M) + t_1(M) - 1  & \mbox{ if } i=0,\\
\min\{ 2t_0(M)-2i,\, t_0(M)+t_1(M) -2i - 1 \} & \mbox{ if } i\geqslant 1.
  \end{array}\right. \]
 Also,
 \[ h_i(M) \leqslant
\left\{ \begin{array}{ll}
 \delta(M) + t_0(M) - 1 & \mbox{ if } i=1, \\
 2\delta(M) -2i +2 & \mbox{ if } i\geqslant 2,
\end{array}\right. \]
where $\delta(M)$ denotes the stable degree of $M$.

\item The Castelnuovo-Mumford regularity of $M$ satisfies:
\begin{align*}
\reg(M) & \leqslant \max\{ h_0(M),\, \min\{ 2 t_0(M)-1,\,
 t_0(M)+t_1(M)-2 \} \};\\
\reg(M) & \leqslant t_0(M) + t_1(M) - 1.
\end{align*}

\item $\Sigma^X M$ is semi-induced if $\dim_F(X) \geqslant t_0(M) + t_1(M)$.

\item If $\kk$ is a field of characteristic 0, then the injective dimension of $M$ satisfies:
\begin{equation*}
\injdim(M) \leqslant \max \{t_0(M), t_0(M) + t_1(M) -1 \}.
\end{equation*}
\end{enumerate}
\end{theorem}

From this theorem one can deduce more bounds, all of which are in terms of $t_0(M)$ and $t_1(M)$, for a finitely generated $\VI$-module $M$; for example:
\begin{itemize}
\item
if $\kk$ is a field, then $\dim_\kk(M(F^n))$ is a polynomial in $q^n$ for $n\geqslant t_0(M)+t_1(M)$;

\item
if $\kk$ is an algebraically closed field of characteristic zero, then the sequence $\{ M(F^n)\}$ of $\kk[\mathbf{GL}_n(F)]$-modules is representation stable in the sense of \cite[Definition 1.5]{gw} starting at $n=\max\{ t_0(M)+t_1(M), 2t_0(M) \}$.

\end{itemize}

This paper is, admittedly, not self-contained since our main goal is to explain how the shift theorem in \cite{nagpal} and the approach in \cite{li2, ly} can be used to get explicit bounds for homological invariants of finitely generated $\VI$-modules; we avoid repeating definitions and preliminary results in those papers unless it is necessary. However, we do try to make it readable for anyone who is somewhat familiar with any of the existing papers on homology of $\FI$-modules or $\VI$-modules. It is recommended to read \cite{nagpal} before turning to the present paper.

The paper is organized as follows. In Section 2, we introduce the notations and prove some preliminary results which are used later in the paper. Formulas for upper bounds of Castelnuovo-Mumford regularity and injective dimension are proved in Section 3. In the last section, we recall the construction of a finite complex given by the shift theorem in \cite{nagpal}; it was shown in \cite{nagpal} that this complex computes local cohomology. Using results from the previous sections, we obtain upper bounds of degrees of local cohomology and certain other invariants.

\section{Generalities}

For convenience and to avoid confusion due to differences in notations, we adopt the notations of \cite{nagpal}, but our convention for the degree of the zero $\VI$-module is $-\infty$. In particular, for any $\VI$-module $M$ and $X\in \Ob(\VI)$, we denote:
\begin{tabbing}
\quad\= $M_n$ \;\quad\=: the $\kk[\mathbf{GL}_n(F)]$-module $M(F^n)$; \\

\> $M_{\prec d}$ \>: the smallest $\VI$-submodule of $M$ containing $M_n$ for all $n<d$; \\

\> $\Gamma(M)$ \>: the maximal torsion submodule of $M$; \\

\> $\RR^i\Gamma$ \>: the $i$-th right derived functor of $\Gamma$; \\

\> $h_i(M)$ \>: the degree of $\RR^i \Gamma(M)$; \\

\> $\HH_0^\VI$ \>: the $\VI$-homology functor defined by $\HH_0^\VI(M)(X) = (M/M_{\prec d})(X)$ where $d=\dim_F (X)$;\\

\> $\HH_i^\VI$ \>: the $i$-th left derived functor of $\HH_0^\VI$; \\

\> $t_i(M)$ \>: the degree of $\HH_i^\VI(M)$; \\

\> $\delta(M)$ \>: the stable degree of $M$; \\

\> $\Sigma^X$ \>: the functor defined by $(\Sigma^X M)(Z)=M(X+Z)$ for every $Z\in\Ob(\VI)$; \\

\> $\Delta^X$ \>: the cokernel of the natural transformation $\mathrm{id} \to \Sigma^X$; \\

\> $\UU_X$ \>: the $\VI$-group defined in \cite[\S4.2]{nagpal};\\

\> $\bar\Sigma^X$ \>: the functor defined by $\bar\Sigma^X M = \left( \Sigma^X M \right)_{\UU_X}$; \\

\> $\bar\Delta^X$ \>: the functor defined by $\bar\Delta^X M  = \left( \Delta^X M \right)_{\UU_X}$;\\

\> $\kappa^X$ \>: the kernel of the natural transformation $\mathrm{id} \to \bar\Sigma^X$;\\

\> $\II(V)$ \>: the $\VI$-module induced from a $\VB$-module $V$,
\end{tabbing}
where $\VB$ is the category whose objects are the finite dimensional $F$-vector spaces and whose morphisms are the bijective $F$-linear maps.

Let us mention that $t_i(M)$ here are called \emph{homological degrees} and denoted by $\mathrm{hd}_i(M)$ in \cite{li2, ly}; in particular, $t_0(M)$ coincides with the \emph{generating degree} $\mathrm{gd}(M)=\mathrm{hd}_0(M)$ of $M$. The degree $h_0(M)$ of $\Gamma(M)$ here is called \emph{torsion degree} $\mathrm{td}(M)$ of $M$ in \cite{li2,ly}.

We remind the reader that finitely generated $\VI$-modules are Noetherian (\cite[Theorem A]{ps} and \cite[Corollary 8.3.3]{ss}).

\begin{lemma}  \label{replace}
Let $M$ be a $\VI$-module presented in finite degrees. If $t_1(M)<t_0(M)$, then $M/M_{\prec d}$ is induced from $d$, where $d=t_0(M)$.
\end{lemma}
\begin{proof}
Let $Q=M/M_{\prec d}$. The short exact sequence $0\to M_{\prec d}\to M\to Q\to 0$ induces a long exact sequence $\cdots \to \HH_1^\VI(M) \to \HH_1^\VI(Q) \to \HH_0^\VI(M_{\prec d}) \to \cdots$, so $t_1(Q)<d$.

Let $P=\II(Q_d)$ and let $N$ be the kernel of the natural epimorphism $P\to Q$. The short exact sequence $0\to N \to P \to Q \to 0$ induces a long exact sequence
\begin{equation*}
\cdots \to \HH_1^\VI(Q) \to \HH_0^\VI (N) \to \HH_0^\VI(P) \to \HH_0^\VI(Q) \to 0.
\end{equation*}
Since the map $\HH_0^\VI(P) \to \HH_0^\VI(Q)$ is an isomorphism, the map $\HH_1^\VI(Q) \to \HH_0^\VI (N)$ is surjective, so $t_0(N)\leqslant t_1(Q)< d$. But $N_n=0$ for every $n< d$. Therefore $N=0$, so $P$ is isomorphic to $Q$.
\end{proof}

Given a VI-module $M$ and $X\in \Ob(\VI)$, we have an exact sequence:
\begin{equation*}
0 \to \kappa^X M \to M \to \bar{\Sigma}^X M \to \bar{\Delta}^X M \to 0.
\end{equation*}
The analogue of this sequence has played a central role in the representation theory of FI, and as we will see, it also plays a crucial role for us to obtain bounds of homological invariant of VI-modules. The next lemmas give us some preliminary results related to this sequence.

\begin{lemma} \label{decrease}
Let $M$ be a $\VI$-module presented in finite degrees. For any nonzero $X\in \Ob(\VI)$, one has: $t_1( \bar\Delta^X M ) \leqslant t_1(M)-1$.
\end{lemma}
\begin{proof}
By Lemma \ref{replace}, if $t_1(M)<t_0(M)$, we may replace $M$ by $M_{\prec d}$ where $d=t_0(M)$. Repeating this if necessary, we may assume that $t_0(M)\leqslant t_1(M)$.

Let $0 \to N \to P \to M \to 0$ be a short exact sequence where $P$ is an induced $\VI$-module with $t_0(P)=t_0(M)$.  Then $t_0(N)\leqslant t_1(M)$. We get an exact sequence
\begin{equation*}
\bar\Delta^X N \to \bar\Delta^X P \to \bar\Delta^X M \to 0.
\end{equation*}
By \cite[Corollary 4.19]{nagpal}, one has $t_0(\bar\Delta^X N) \leqslant t_0(N)-1$. Let $K$ be the kernel of the map $\bar\Delta^X P \to \bar\Delta^X M$. Then
\begin{equation*}
t_0(K) \leqslant t_0(\bar\Delta^X N) \leqslant t_0(N)-1 \leqslant t_1(M)-1.
\end{equation*}
The short exact sequence $0\to K \to \bar\Delta^X P \to \bar\Delta^X M \to 0$ gives a long exact sequence
\begin{equation*}
\cdots \to \HH_1^\VI(\bar\Delta^X P) \to \HH_1^\VI (\bar\Delta^X M ) \to \HH_0^\VI (K) \to \cdots.
\end{equation*}
Since $\bar\Delta^X P$ is induced, we have $\HH_1^\VI(\bar\Delta^X P)=0$, so
$t_1(\bar\Delta^X M) \leqslant t_0(K) \leqslant t_1(M)-1$.
\end{proof}

As in \cite{nagpal}, we drop the superscript $X$ in $\Sigma^X$, $\Delta^X$, $\bar\Sigma^X$, $\bar\Delta^X$ and $\kappa^X$ when $\dim_F (X)=1$.

By definition, $\kappa^X M$ is the kernel of the natural map $M \to \bar{\Sigma}^X M$. The next lemma asserts that $\kappa^X M$ coincides with the kernel of another natural map $M \to \Sigma^X M$.

\begin{lemma} \label{kernel}
For any $X\in \Ob(\VI)$, the kernel of the natural transformation $\mathrm{id} \to \Sigma^X$ is $\kappa^X$.
\end{lemma}
\begin{proof}
Let $M$ be any $\VI$-module. The $\UU_X$-action on $M$ is trivial and the natural map $M \to \Sigma^X M$ is $\UU_X$-equivariant. Since the order of the group $\UU_X(Z)$ is invertible in $\kk$ for every $Z\in \Ob(\VI)$, the lemma follows.
\end{proof}

We now consider $\kappa M$. By the previous lemma, $\kappa M$ is also the kernel of the natural map $M \to \Sigma M$. Explicitly, for any object $Z$ of $\VI$, the value $(\kappa M) (Z)$ of $\kappa M$ on $Z$ consists of those elements $v \in M(Z)$ such that $\iota_* (v) = 0$ for the standard inclusion $\iota: Z \to X + Z$ (where $X+Z$ is the direct sum of $Z$ with the one-dimensional $F$-vector space $X$ used to define the shift functor $\Sigma$). Since the group $\mathbf{GL} (X + Z)$ acts transitively from the left on the set of morphisms from $Z$ to $X+Z$, one knows that:
\begin{equation*}
(\kappa M)(Z) = \{ v \in M(Z) \mid f_* (v)  = 0 \mbox{ for every morphism } f: Z \to X+Z \}.
\end{equation*}
Therefore, $(\kappa M)(Z)$ is a $\VI$-submodule of $M$, and one has the following decomposition:
\begin{equation*}
\kappa M = \bigoplus _{Z \in \Ob (\mathrm{VI})} (\kappa M)(Z)
\end{equation*}
as VI-modules. From this observation, one deduces the following simple but key fact. (A similar fact for FI-modules was shown in \cite[Lemma 2.1]{li2}.)

\begin{lemma} \label{torsion degree}
For any $\VI$-module $M$, one has $h_0(M) = t_0(\kappa M)$.
\end{lemma}
\begin{proof}
By Lemma \ref{kernel} and the above observation, one has $h_0(M) = \deg(\kappa M) = t_0(\kappa M)$.
\end{proof}

We also need the following technical lemma to establish our main results.

\begin{lemma} \label{semi-induced}
Let $M$ be a $\VI$-module and $X\in\Ob(\VI)$. If $\bar\Sigma^X M$ is semi-induced, then:
\begin{align*}
\kappa^X M &= \Gamma(M),\\
t_1(\bar\Delta^X M) &\leqslant t_0(M),\\
t_{i+1} (\bar\Delta^X M) &= t_i (M/\Gamma(M)) \quad\mbox{ for each } i\geqslant 1.
\end{align*}
\end{lemma}
\begin{proof}
 By \cite[Proposition 4.27 and Corollary 4.22]{nagpal}, we have $\bar\Sigma^X \Gamma(M) = \Gamma (\bar\Sigma^X M) = 0$. Hence,  $\kappa^X \Gamma(M) = \Gamma(M)$. Since $M/\Gamma(M)$ is torsion-free, the map $M/\Gamma(M) \to \Sigma^X (M/\Gamma(M))$ is injective, so by \cite[Lemma 4.7]{nagpal}, the map $M/\Gamma(M) \to \bar\Sigma^X ( M/\Gamma(M) )$ is injective. Hence, $\kappa^X (M/\Gamma(M)) = 0$. Since $\kappa^X$ is a left exact functor, it follows that $\kappa^X M = \Gamma (M)$.

By \cite[Proposition 3.10]{nagpal}, we know that $\bar\Sigma^X M$ is homology acyclic. The short exact sequence
\begin{equation*}
0\to M/\Gamma(M) \to \bar\Sigma^X M \to \bar\Delta^X M\to 0\
\end{equation*}
gives a long exact sequence:
\begin{equation*}
\cdots  \to 0 \to \HH_2^\VI(\bar\Delta^X M) \to \HH_1^\VI(M/\Gamma(M)) \to 0 \to \HH_1^\VI (\bar\Delta^X M) \to \HH_0^\VI(M/\Gamma(M)) \to \cdots.
\end{equation*}
Hence, $t_1(\bar\Delta^X M)\leqslant t_0(M/\Gamma(M))\leqslant t_0(M)$, and $t_{i+1} (\bar\Delta^X M) = t_i (M/\Gamma(M))$ for each $i\geqslant 1$.
\end{proof}

\section{Bounds on regularity and injective dimensions}

The main task of this section is to prove formulas for upper bounds of Castelnuovo-Mumford regularity and, when $\kk$ is a field of characteristic zero, a formula for the injective dimension. First, let us make the following definition:

\begin{definition}
For any $\VI$-module $M$, the \emph{Castelnuovo-Mumford regularity} of $M$, denoted by $\reg(M)$, is the supremum of $\{t_i(M)-i : i\geqslant 1\}$.
\end{definition}

An upper bound (in terms of $t_0(M)$ and $t_1(M)$) of Castelnuovo-Mumford regularity of FI-modules presented in finite degrees was first obtained by Church and Ellenberg in \cite{ce}. In \cite{li2}, the second author provided a new proof as well as another upper bound in terms of $t_0(M)$ and $h_0(M)$. Subsequently, an alternative proof was given by the first author in \cite{gan}. Motivated by the work of \cite{li2} and \cite{ly}, in a short note \cite{church}, Church simplified the original proof in \cite{ce}. The following theorem asserts that all these upper bounds for FI-modules also hold for finitely generated VI-modules. Its proof uses the shift theorem of Nagpal \cite{nagpal} for finitely generated $\VI$-modules; we do not know if the shift theorem and our results below can be extended to $\VI$-modules presented in finite degrees.

\begin{theorem} \label{main theorem}
Let $M$ be a finitely generated $\VI$-module. Then:
\begin{align*}
h_0(M) &\leqslant t_0(M) + t_1(M) - 1,\\
\reg(M) &\leqslant \max\{ h_0(M),\, \min\{ 2 t_0(M)-1,\,
 t_0(M)+t_1(M)-2 \} \},\\
\reg(M) &\leqslant t_0(M) + t_1(M) - 1.
\end{align*}
\end{theorem}
\begin{proof}
If $M=0$, then $h_0(M)$ and $\reg(M)$ are both $-\infty$, so there is nothing to prove. Suppose $M\neq 0$. We use induction on $t_0(M)$.

By \cite[Proposition 3.8]{nagpal}, there is a short exact sequence $0 \to N \to P \to M \to 0$ such that $P$ is an induced $\VI$-module with $t_0(P)=t_0(M)$. By \cite[Propositions 4.8 and 4.21]{nagpal}, we have $\kappa P = 0$. Therefore, by \cite[Proposition 4.17]{nagpal}, we have an exact sequence:
\begin{equation*}
0 \to \kappa M \to \bar\Delta N \to \bar\Delta P \to \bar\Delta M \to 0.
\end{equation*}
We break this into two short exact sequences:
\begin{gather*}
0 \to \kappa M \to \bar\Delta N \to \bar\Delta N / \kappa M  \to 0,\\
0 \to \bar\Delta N / \kappa M  \to  \bar\Delta P \to \bar\Delta M \to 0.
\end{gather*}
They give two long exact sequences:
\begin{equation} \label{main les 1}
\cdots \to \HH_1^\VI (\bar\Delta N/ \kappa M ) \to \HH_0^\VI (\kappa M) \to \HH_0^\VI (\bar\Delta N) \to \cdots,
\end{equation}
\begin{equation} \label{main les 2}
\cdots  \to \HH_2^\VI (\bar\Delta M) \to \HH_1^\VI (\bar\Delta N /\kappa M) \to 0 \to \cdots,
\end{equation}
where we used \cite[Propositions 3.4 and 4.5]{nagpal} to see that $\HH_1^\VI (\bar\Delta P) = 0$. By \cite[Corollary 4.19]{nagpal}, we have:
\begin{gather}
t_0(\bar\Delta N) \leqslant t_0(N)-1 \leqslant \max\{t_0(M), t_1(M)\} -1, \label{ineq 1}\\
t_0(\bar\Delta M) \leqslant t_0(M)-1. \label{ineq 2}
\end{gather}
By Lemma \ref{decrease}, we have:
\begin{equation}
t_1(\bar\Delta M) \leqslant t_1(M)-1. \label{ineq 3}
\end{equation}
Hence,
\begin{align*}
h_0(M)
&= t_0(\kappa M)  & \mbox{by Lemma \ref{torsion degree}} \\
&\leqslant \max\{ t_0(\bar\Delta N), t_1(\bar\Delta N / \kappa M)\} & \mbox{by \eqref{main les 1}}   \\
&\leqslant \max\{ t_0(M)-1,\, t_1(M)-1,\, t_2(\bar\Delta M)  \}   & \mbox{by \eqref{main les 2} and \eqref{ineq 1}} \\
&\leqslant \max\{ t_0(M)-1,\, t_1(M)-1,\, \reg(\bar\Delta M) + 2 \} \\
&\leqslant \max\{ t_0(M)-1,\, t_1(M)-1,  t_0(\bar\Delta M) + t_1(\bar\Delta M) + 1 \} & \mbox{by induction hypothesis}\\
&\leqslant \max\{ t_0(M)-1,\, t_1(M)-1,\, t_0(M)+t_1(M) -1 \}  & \mbox{by \eqref{ineq 2} and \eqref{ineq 3}} \\
&\leqslant t_0(M) + t_1(M) -1.
\end{align*}

Next, by \cite[Theorem 4.34(a)]{nagpal}, if $X$ is a finite dimensional $F$-vector space whose dimension is sufficiently large, then $\bar\Sigma^X M$ is semi-induced; we choose and fix such a nonzero $X$. From the long exact sequence associated to the short exact sequence $0\to \Gamma(M) \to M \to M/\Gamma(M) \to 0$, we see that:
\begin{equation*}
\reg(M) \leqslant \max\{ \reg(\Gamma(M)),  \reg(M/\Gamma(M)) \}.
\end{equation*}
By \cite[Corollary 5.1]{li1}, we have:
\begin{equation*}
\reg( \Gamma(M) ) \leqslant h_0(M).
\end{equation*}
By \cite[Corollary 4.19]{nagpal}, we have:
\begin{equation} \label{ineq 5}
t_0(\bar\Delta^X M) \leqslant t_0(M)-1.
\end{equation}
Hence,
\begin{align*}
\reg( M/\Gamma(M)) & \leqslant \reg(\bar\Delta^X M) + 1 & \mbox{by Lemma \ref{semi-induced}} \\
& \leqslant t_0(\bar\Delta^X M) + t_1(\bar\Delta^X M)  & \mbox{by induction hypothesis}  \\
& \leqslant \min\{ 2 t_0(M)-1,\,
 t_0(M)+t_1(M)-2 \} & \mbox{by Lemma \ref{decrease}, Lemma \ref{semi-induced} and \eqref{ineq 5}}.
\end{align*}
It follows from the above that:
\begin{equation} \label{ineq 6}
\reg(M) \leqslant \max\{ h_0(M),\, \min\{ 2 t_0(M)-1,\,
 t_0(M)+t_1(M)-2 \} \}.
\end{equation}
Finally, since we have $h_0(M) \leqslant t_0(M) + t_1(M) -1$,
it follows from \eqref{ineq 6} that
\[ \reg(M) \leqslant t_0(M) + t_1(M) - 1. \]
\end{proof}

\begin{remark}
The finiteness of Castelnuovo-Mumford regularity of finitely generated VI-modules in non-describing characteristic case was proved by Nagpal in \cite[Theorem 5.13]{nagpal}. Furthermore, he gave an upper bound of this invariant in terms of the degrees of the local cohomology of $M$ which in general are hard to obtain.
\end{remark}

\begin{remark}
Let us also mention that for a closely related category $\mathrm{VIC}$, Miller and Wilson \cite[Theorem 2.26]{mw} proved an upper bound for the degrees of syzygies of $\mathrm{VIC}$-modules over a field of characteristic zero; their bound grow exponentially.
\end{remark}

In the rest of this section, we study the injective dimension of finitely generated VI-modules when $\kk$ is a field of characteristic zero. For any $\VI$-module $M$, denote by $\injdim(M)$ the injective dimension of $M$. We adopt the convention that the zero module has injective dimension $-\infty$.

First, we consider the special case of finitely generated torsion $\VI$-modules $M$; in this special case, there is a non-negative integer $N$ such that $M(Z)=0$ for all objects $Z$ with $\dim_F(Z) > N$.

\begin{lemma} \label{id for torsion}
Assume that $\kk$ is a field of characteristic zero. Let $M$ be a finitely generated torsion $\VI$-module. Then $\injdim (M) \leqslant \deg(M)$.
\end{lemma}
\begin{proof}
As shown in the proof of \cite[Lemma 3.2]{glx}, there exists an exact sequence
\begin{equation*}
0 \to M \to E^0 \to E^1 \to \cdots \to E^n \to 0
\end{equation*}
where each $E^i$ is a finitely generated torsion injective $\VI$-module such that $\deg(E^i) \leqslant \deg(M)-i$. Since $E^i=0$ if $i>\deg(M)$, the exact sequence implies that $\injdim (M)\leqslant \deg(M)$.
\end{proof}

Now we are ready to extend the result of \cite[Corollary 3.7]{li2} for FI-modules to VI-modules.

\begin{corollary}
Assume that $\kk$ is a field of characteristic zero. Let $M$ be a finitely generated $\VI$-module. Then:
\begin{align}
\injdim (M) &\leqslant \max\{ h_0(M),\, t_0(M),\, 2 t_0(M) -1  \},  \label{id1} \\
\injdim (M) &\leqslant \max\{ t_0(M),\, t_0(M) + t_1(M) - 1 \}.  \label{id2}
\end{align}
\end{corollary}
\begin{proof}
We use induction on $t_0(M)$. The base case of the zero module is trivial. Suppose $M$ is nonzero. We have:
\begin{equation*}
\injdim(M) \leqslant \max\{  \injdim (\Gamma(M)) ,\,  \injdim (M / \Gamma(M)) \}.
\end{equation*}
By Lemma \ref{id for torsion},
\begin{equation*}
\injdim (\Gamma(M)) \leqslant h_0(M).
\end{equation*}
By \cite[Theorem 4.34(a)]{nagpal}, we can choose a nonzero $X\in \Ob(\VI)$ such that $\bar\Sigma^X M$ is semi-induced. Since $\kk$ is a field characteristic zero, $\bar\Sigma^X M$ is injective (by \cite[Theorem 1.5]{gl}). By Lemma \ref{semi-induced}, we have a short exact sequence  $0\to M/\Gamma(M) \to \bar\Sigma^X M \to \bar\Delta^X M \to 0$. Hence,
\begin{align*}
\injdim (M/\Gamma(M))  &\leqslant \max\{  0, \injdim (\bar\Delta^X M) + 1  \} \\
&\leqslant \max\{ t_0(M),  \injdim (\bar\Delta^X M) + 1 \}.
\end{align*}
By \cite[Corollary 4.19]{nagpal}, $t_0(\bar\Delta^X M) \leqslant t_0(M)-1$. By Lemma \ref{semi-induced}, $t_1(\bar\Delta^X M) \leqslant t_0(M)$. By Theorem \ref{main theorem},
\begin{equation*}
h_0(\bar\Delta^X M) \leqslant t_0(\bar\Delta^X M) + t_1(\bar\Delta^X M) -1 \leqslant 2 t_0(M)  - 2.
\end{equation*}
Hence, by induction hypothesis, we have:
\begin{align*}
\injdim (\bar\Delta^X M) +1 & \leqslant \max\{ h_0(\bar\Delta^X M), t_0(\bar\Delta^X M),  2t_0(\bar\Delta^X M) - 1 \} +1 \\
&\leqslant \max\{ t_0(M), 2t_0(M) - 1 \}.
\end{align*}
Putting the above inequalities together, we deduce \eqref{id1}.

By induction hypothesis, we also have:
\begin{align*}
\injdim (\bar\Delta^X M) +1 & \leqslant \max\{
t_0(\bar\Delta^X M) + 1, t_0(\bar\Delta^X M) + t_1(\bar\Delta^X M)\} \\
& \leqslant \max\{ t_0(M), t_0(M) + t_1(M) -2 \},
\end{align*}
where we used Lemma \ref{decrease}. By Theorem \ref{main theorem}, we have $h_0(M) \leqslant t_0(M)+t_1(M)-1$. Hence, \eqref{id2} follows.
\end{proof}

\begin{remark}
The finiteness of injective dimension of finitely generated VI-modules over fields of characteristic zero was proved in \cite[Theorem 5.25]{nagpal} but no explicit upper bound was given there.
\end{remark}

\section{Bounds on local cohomology}

In this section, we give upper bounds for degrees of the local cohomology of finitely generated VI-modules. We first recall, for any finitely generated $\VI$-module $M$, the construction of a complex $I^\bullet$ in the proof of \cite[Theorem 4.34(b)]{nagpal}.

Let $I^0=M$ and let $C^0=M$. For each $i \geqslant 1$:
\begin{itemize}
\item
choose a nonzero $X_i\in \Ob(\VI)$ such that $\bar\Sigma^{X_i} C^{i-1}$ is semi-induced;

\item
let $I^i = \bar\Sigma^{X_i} C^{i-1}$ and let $C^i = \bar\Delta^{X_i} C^{i-1}$;

\item
let $d^{i-1} : I^{i-1}\to I^i$ be the composition of the natural maps: $I^{i-1} \to C^{i-1} \to I^i$.
\end{itemize}
The existence of $X_1, X_2, \ldots$ is guaranteed by \cite[Theorem 4.34(a)]{nagpal}. By construction, $I^i$ is semi-induced for every $i\geqslant 1$. By Lemma \ref{semi-induced}, we have an exact sequence
\begin{equation*}
0 \to \Gamma(C^{i-1}) \to  C^{i-1} \to I^i \to C^i \to 0.
\end{equation*}
Hence, $d^i \circ d^{i-1} = 0$ and $\HH^i(I^\bullet) = \Gamma(C^i)$.

We collect here some inequalities which we shall use in the proof of Theorem \ref{complex bound} below.
By \cite[Corollary 4.19]{nagpal}, we have
\begin{equation} \label{ineq 10}
t_0 (C^i) \leqslant t_0(M)-i \quad \mbox{ for each } i\geqslant 0.
\end{equation}
Hence, $C^i=0$ when $i> t_0(M)$, and $I^i = 0$ when $i>t_0(M)+1$. In particular, $I^\bullet$ is a finite complex. By Lemma \ref{decrease}, we also have $t_1(C^i) \leqslant t_1(C^{i-1})-1$ for each $i\geqslant 1$,  and hence
\begin{equation} \label{ineq for t1 of c}
t_1(C^i) \leqslant t_1(M)-i \quad \mbox{ for each } i\geqslant 0.
\end{equation}
By \cite[Proposition 5.1]{nagpal}, we have:
\begin{equation} \label{bound with delta}
t_0(C^i) \leqslant t_0(I^i) = \delta(I^i) = \delta(C^{i-1})
\leqslant \delta(M) - i+1 \quad \mbox{ for each } i \geqslant 1.
\end{equation}

Each cohomology $\HH^i(I^\bullet)$ is a $\VI$-module. The following theorem establishes an upper bound for the degree of the $\VI$-module $\HH^i(I^\bullet)$. The FI analogue of Theorem \ref{complex bound}(a) was proved in \cite[Proposition 3.1]{li2} (with a slight difference because a different indexing was used there), and the FI analogue of Theorem \ref{complex bound}(b) was proved in \cite[Theorem 2.10]{cmnr}.

\begin{theorem} \label{complex bound}
Let $M$ be a finitely generated $\VI$-module, and let $I^\bullet$ be the complex constructed above. Then:
\begin{enumerate}[(a)]
\item
\begin{align*}
\deg \HH^0(I^\bullet) &= h_0(M),\\
\deg \HH^i(I^\bullet) &\leqslant
\min\{ 2t_0(M)-2i,\, t_0(M)+t_1(M) -2i - 1 \}\quad \mbox{ for each } i\geqslant 1.
\end{align*}

\item
\[
\deg \HH^i(I^\bullet) \leqslant
\left\{ \begin{array}{ll}
 \delta(M) + t_0(M) - 1 & \mbox{ if } i=1, \\
 2\delta(M) -2i +2 & \mbox{ if } i\geqslant 2.
\end{array}\right.
\]
\end{enumerate}
\end{theorem}
\begin{proof}
(a)
We have $\deg \HH^0(I^\bullet) = \deg \Gamma(C^0) = h_0(M)$. For each $i\geqslant 1$, we have:
\begin{align*}
\deg \HH^i(I^\bullet) &= \deg \Gamma(C^i)  \\
&= h_0(C^i)   \\
&\leqslant t_0(C^i) + t_1(C^i) - 1 \qquad \mbox{(by Theorem \ref{main theorem}).}
\end{align*}
By Lemma \ref{semi-induced} and \eqref{ineq 10}, we have:
\begin{gather*}
t_0(C^i) \leqslant t_0(M)-i,  \\
t_1(C^i) \leqslant t_0(C^{i-1}) \leqslant t_0(M) - i+1.
\end{gather*}
Hence,
\[ \deg \HH^i(I^\bullet) \leqslant (t_0(M)-i) + (t_0(M)-i+1) - 1 = 2t_0(M)-2i. \]
Also, by \eqref{ineq 10} and \eqref{ineq for t1 of c},
\[ \deg \HH^i(I^\bullet) \leqslant (t_0(M) - i) + (t_1(M) - i) -1 = t_0(M) + t_1(M) -2i -1. \]

(b) By above and \eqref{bound with delta}, we have:
\begin{align*}
\deg \HH^i(I^\bullet) &\leqslant t_0(C^i) + t_0(C^{i-1}) -1 \\
&\leqslant
\left\{ \begin{array}{ll}
 \delta(M) + t_0(M) - 1 & \mbox{ if } i=1, \\
 2\delta(M) -2i +2 & \mbox{ if } i\geqslant 2.
\end{array}\right.
\end{align*}
\end{proof}

Nagpal \cite{nagpal} proved that the complex $I^\bullet$ computes the local cohomology of $M$. (For $\FI$-modules, this was proved by Ramos and the second author in \cite{lr}.) Therefore, one has:

\begin{corollary}
Let $M$ be a finitely generated $\VI$-module. Then:
\begin{align*}
h_i(M) & \leqslant
\min\{ 2t_0(M)-2i,\, t_0(M)+t_1(M) -2i - 1 \}\quad \mbox{ for each } i\geqslant 1,\\
h_i(M) & \leqslant
\left\{ \begin{array}{ll}
 \delta(M) + t_0(M) - 1 & \mbox{ if } i=1, \\
 2\delta(M) -2i +2 & \mbox{ if } i\geqslant 2.
\end{array}\right.
\end{align*}
\end{corollary}
\begin{proof}
By \cite[Corollary 5.10]{nagpal}, we have $\RR^i \Gamma(M) = \HH^i (I^\bullet)$. The corollary is now immediate from Theorem \ref{complex bound}.
\end{proof}

\begin{remark}
The finiteness of $h_i(M)$ was established in \cite[Theorem 5.9]{nagpal}. In \cite[Question 1.12]{nagpal}, Nagpal asked whether there exists an upper bound of $h_i(M)$ in terms of $t_0(M)$ and $t_1(M)$. The above corollary answers this question for finitely generated modules.
\end{remark}

By \cite[Remark 5.5]{nagpal}, the $\VI$-module $\Sigma^X M$ is semi-induced if and only if $\dim_F(X) \geqslant \max \{ h_i(M) + 1 \mid i \geqslant 0 \}$, which is a finite number since $h_i(M) = -\infty$ when $i \gg 0$. Since it is not easy to find $h_i(M)$, the following corollary is more useful in practice.

\begin{corollary} \label{bound for shift}
Let $M$ be a finitely generated $\VI$-module and $X\in \Ob(\VI)$. Then $\Sigma^X M$ is semi-induced if
\begin{equation} \label{lower bound 1}
\dim_F(X)  \geqslant  \max\{h_0(M) + 1, \, \min\{ 2t_0(M) - 1 , \, t_0(M) + t_1(M) -2 \} \},
\end{equation}
or if
\begin{equation} \label{lower bound 2}
\dim_F(X) \geqslant t_0(M) + t_1(M).
\end{equation}
\end{corollary}
\begin{proof}
First, suppose that \eqref{lower bound 1} holds. By Theorem \ref{complex bound}, the complex $\Sigma^X I^\bullet$ is exact. For each $i\geqslant 1$, the $\VI$-module $I^i$ is semi-induced, so by \cite[Corollary 4.4]{nagpal}, the $\VI$-module $\Sigma^X I^i$ is semi-induced. By \cite[Corollary 4.23]{nagpal}, it follows that $\Sigma^X M=\Sigma^X I^0$ is semi-induced.

Next, suppose that \eqref{lower bound 2} holds. By Theorem \ref{main theorem}, we have $t_0(M)+t_1(M) \geqslant h_0(M)+1$, so we are done by \eqref{lower bound 1}.
\end{proof}

\begin{remark} \label{polynomial growth}
Suppose that $\kk$ is a field. Let $M$ be a finitely generated $\VI$-module. It was shown in the proof of \cite[Theorem 5.4]{nagpal} that if $X\in \Ob(\VI)$ and $\Sigma^X M$ is semi-induced, then there is a polynomial $P$ of degree $t_0(\Sigma^X M)$  such that $\dim_\kk (M_n) = P(q^n)$ if  $n\geqslant \dim_F(X)$. By Corollary \ref{bound for shift}, we deduce that $\dim_\kk (M_n) = P(q^n)$ if  $n  \geqslant t_0(M) + t_1(M)$.
\end{remark}

\begin{remark} \label{rep stability}
Suppose that $\kk$ is an algebraically closed field of characteristic zero.  Let $M$ be a finitely generated $\VI$-module, and choose $X\in \Ob(\VI)$ such that $\Sigma^X M$ is semi-induced. The proof of \cite[Theorem 5.14]{nagpal} shows that $M$ (or more precisely, the consistent sequence $M_0\to M_1 \to M_2 \to \cdots$ determined by $M$ and the standard inclusion maps $F^0 \hookrightarrow F^1 \hookrightarrow F^2 \hookrightarrow \cdots$) is representation stable in the sense of \cite[Definition 1.5]{gw} starting at $\max\{ \dim_F(X), 2t_0(M)  \}$. By Corollary \ref{bound for shift}, we deduce that $M$ is representation stable starting at  $\max\{ t_0(M) + t_1(M), 2t_0(M) \}$. We point out that the proof of \cite[Theorem 5.14]{nagpal} uses Pieri's formula which shows that any induced module $\II(V)$ is representation stable starting at $2\deg(V)$.

\end{remark}

\end{document}